\documentclass[11pt]{amsart}

\usepackage{amssymb, graphicx}

\usepackage{amsfonts}
\usepackage{latexsym}
\usepackage{amscd}

\allowdisplaybreaks[1]

\newtheorem{theorem}{Theorem}[section]

\newtheorem{lemma}[theorem]{Lemma}
\newtheorem{corollary}[theorem]{Corollary}
\theoremstyle{remark}
\newtheorem{remark}[theorem]{Remark}
            
\newtheorem{example}[theorem]{Example}   
\newtheorem{problem}[theorem]{Problem}

\newcommand{\s}{\sigma} 
\newcommand{\CC}{\mathbb{C}}
\newcommand{\R}{\mathbb{R}}
 
\begin{document}
\title[Determining elements  through  spectral properties]{Determining elements in Banach algebras through spectral properties}

\author{Matej Bre\v sar}
\author{\v Spela \v Spenko}
\address{M. Bre\v sar,  Faculty of Mathematics and Physics,  University of Ljubljana,
 and Faculty of Natural Sciences and Mathematics, University
of Maribor, Slovenia} \email{matej.bresar@fmf.uni-lj.si}
\address{\v S. \v Spenko,  Institute of  Mathematics, Physics, and Mechanics,  Ljubljana, Slovenia} \email{spela.spenko@imfm.si}

\begin{abstract} 
Let $A$ be a Banach algebra. By $\sigma(x)$ and $r(x)$ we denote the spectrum and the spectral radius of $x\in A$, respectively. 
We consider the relationship between elements $a,b\in A$ that satisfy one of the following two conditions: (1) $\sigma(ax) = \sigma(bx)$ for all $x\in A$,  (2)  $r(ax) \le r(bx)$ 
for all $x\in A$. In particular we show that (1) implies $a=b$ if $A$ is a $C^*$-algebra, and (2) implies $a\in \mathbb C b$ if $A$ is a prime $C^*$-algebra. As an application of 
the results concerning the conditions 
(1) and (2) we obtain some spectral characterizations of multiplicative maps.
\end{abstract}

\thanks{2010 {\em Math. Subj. Class.} 46H05, 46J05, 46L05, 47A10. }
\thanks{Supported by ARRS Grant P1-0288.}
\maketitle
\section{Introduction}

By a Banach algebra we shall mean a complex Banach algebra. For simplicity of the exposition we assume that all our algebras have  identity elements. 
The spectrum of an element $a$  of a Banach algebra $A$ will be denoted by  $\sigma(a)$, or, occasionally, by  $\sigma_A(a)$. By $r(a)$ we denote the spectral radius of $a$. We write $Z(A)$ for the center of $A$.

Recall that a Banach algebra $A$ is {\em semisimple} if and only if the only element $a\in A$ with the property 
$\sigma(ax)=\{0\}$ for all $x\in A$ is the zero element. That is, $\sigma(ax)=\sigma(0x)$ for all $x\in A$ implies $a=0$.
We propose to study the following problem.

\begin{problem}\label{P}
Let  $A$ be a  semisimple Banach algebra.
Suppose that $a,b\in A$ satisfy 
\begin{equation}\label{eab}
\sigma(ax) = \sigma(bx)\quad\mbox{ for all $x\in A$}.
\end{equation}
Does this imply $a=b$? 
\end{problem}

We do not know the answer in general. In various special cases, however, we are able to show that it is affirmative. Firstly, we will establish  this under the assumption that $a$ can be written as the product of an idempotent and an invertible element. The proof is based on a spectral characterization of 
central idempotents, which may be  of independent interest. 
  Secondly, we will handle the case where $A$ is a commutative Banach algebra, and thirdly, in the main result of Section \ref{S2}, the case where $A$ is a $C^*$-algebra. 
%We are able to solve Problem \ref{P} in four special cases:  (a) $A$ is commutative, (b) $A$ is a $C^*$-algebra,  (c) $a$ is invertible, (d) $a$ is an idempotent.

In Section \ref{S3} we will treat a considerably more general condition that concerns the spectral radius.

\begin{problem}\label{P2}
Let  $A$ be a  semisimple Banach algebra.
Suppose that $a,b\in A$ satisfy 
\begin{equation}\label{sr}
r(ax) \le  r(bx)\quad\mbox{ for all $x\in A$}.
\end{equation}
What is the relation between  $a$ and $b$? 
\end{problem}
This problem is admittedly stated vaguely. However,  we will  see that the answer to our question may depend on the algebra or on the elements in question.
%However, it is difficult to expect that there is  a simple  condition connecting $a$ and $b$ that is characteristic for \eqref{sr} in general semisimple Banach algebras. 
%We will first record an easily proven result concerned with the case where $b$ is invertible. 
A special situation where $b=1$ has been examined earlier by Ptak \cite{Ptak} (and, independently, also in the recent paper \cite{BBR}). The conclusion in this case is that $a\in Z(A)$. 
Our main result concerning Problem \ref{P2} says that  if $A$ is a prime $C^*$-algebra, then the elements $a$ and $b$ satisfying \eqref{sr} are necessarily linearly dependent.

We believe that Problems \ref{P} and \ref{P2} are interesting and challenging in their own right. Our initial motivation for their consideration, however, 
were certain questions centered around Kaplansky's problem on spectrum preserving maps \cite{K}. They are the topic of Section 4. Using the results of Section 2 we will first consider
the problem whether a map  $\varphi$ between Banach algebras $A_0$ and $A$ that satisfies 
\begin{equation}\label{es}
\sigma\bigl(\varphi(x)\varphi(y)\varphi(z)\bigr) = \sigma(xyz)\quad\mbox{for all $x,y,z\in A_0$}
\end{equation}
 is  multiplicative (up to a product with a central element). 
Here we were primarily motivated by Molnar's  paper \cite{Mol} in which he studied a more entangled condition $\sigma\bigl(\varphi(x)\varphi(y)\bigr) = \sigma(xy)$, but only on some special algebras.   
Finally we  will apply  the main result of Section 3 to a  map $\varphi:A_0\to A$ satisfying
  \begin{equation}\label{er}
r\bigl(\varphi(x)\varphi(y)\varphi(z)\bigr) = r(xyz) \quad\mbox{for all $x,y,z\in A_0$.}
\end{equation} 
A somewhat more detailed explanation about the background and motivation for considering \eqref{es} and \eqref{er}
 will be given at the beginning of Section 4.
 
\section{The condition $\sigma(ax) = \sigma(bx)$} \label{S2}

This section is devoted to Problem \ref{P}. To get some feeling for the subject we start by mentioning that in $A=B(X)$, the algebra of all bounded linear operators on a Banach space $X$, \eqref{eab} indeed implies   $a=b$. One just has to take an arbitrary rank one operator for $x$ in \eqref{eab}, and the desired conclusion  easily follows (cf. \cite[Lemma 1]{TL}). In more general Banach algebras, where we do not have appropriate analogues of finite rank operators, the spectrum is not so easily tractable and more sophisticated methods are necessary.

\subsection{Spectral characterization of central idempotents} We begin by recording an elementary lemma which will be needed in the proofs of Theorems \ref{T1} and \ref{Tr}. 
 \begin{lemma} \label{Lld}
 Let $X$ be a complex vector space and let $S,T:X\to X$ be linear operators such that $S\xi\in \mathbb C T\xi$ for every $\xi\in X$. Then $S\in \mathbb C T$.
 \end{lemma}

\begin{proof}
This lemma can be proved directly by elementary methods. On the other hand, one can apply a more general result
 \cite[Theorem 2.3]{BS} which reduces the problem to an easily handled situation where both $S$ and $T$ have rank one.
\end{proof}

In our first theorem we consider a variation of the condition \eqref{eab}.

\begin{theorem}\label{T1}
Let  $A$ be a   semisimple Banach algebra. The following conditions are equivalent for $e\in A$:
\begin{enumerate}
\item[(i)] $ \sigma(ex)\subseteq \sigma(x)\cup\{0\}$ for all $x\in A$.
\item[(ii)] $e$ is a central idempotent. 
\end{enumerate}
\end{theorem}

\begin{proof} (i)$\Longrightarrow$(ii).
Let $\pi$ be an irreducible representation of $A$ on a Banach space $X$. Suppose there exists $\xi \in X$ such that $\xi$ and $\eta =\pi(e)\xi$ are linearly independent.  By Sinclair's extension of the Jacobson density theorem
\cite[Corollary 4.2.6]{Ab} there exists an invertible $t\in A$ such that $\pi(t)\xi = -\eta$ and $\pi(t)\eta = \xi$. Accordingly, 
$$
\pi\bigl( et^{-1}et\bigr)\eta = 
\pi(e)\pi(t)^{-1}\pi(e)\pi(t)\eta = - \eta.
$$
Hence
$$
-1\in \sigma\bigl(\pi\bigl( et^{-1}et\bigr)\Bigr)\subseteq \sigma\bigl( et^{-1}et\bigr)\subseteq  \sigma \bigl(t^{-1}et\bigr)\cup\{0\} = \sigma(e)\cup\{0\} \subseteq \sigma(1)\cup\{0\} = \{0,1\},
$$
a contradiction. Therefore $\pi(e)\xi\in \mathbb C\xi$ for every $\xi\in X$. Lemma \ref{Lld} implies that there exists $\lambda\in\mathbb C$ such that
$\pi(e)=\lambda\pi(1)$. Thus $\lambda\in \sigma\bigl(\pi(e)\bigr) \subseteq \sigma(e)\subseteq \{0,1\}$, so that $\lambda =0$ or $\lambda =1$. Therefore $\pi(e^2) =\pi(e)$
and also $\pi(ex- xe) =0$ for every $x\in A$. The semisimplicity of $A$ implies that $e$ is an idempotent lying in  the center of $A$. 

(ii)$\Longrightarrow$(i). Take  $\lambda\notin \sigma(x)$ such that $\lambda\ne 0$. Then $ex - \lambda$ has an inverse, namely 
$$(ex - \lambda)^{-1}=e(x-\lambda)^{-1} - \lambda^{-1}(1-e).$$
Therefore  $\lambda\notin \sigma(ex)$.
%We begin by noticing that $\sigma(b)\subseteq \{0,1\}$. 
\end{proof}

\begin{corollary}\label{C1}
Let  $A$ be a   semisimple Banach algebra. If $e\in A$ is such that 
 $$\sigma(ex)\cup\{0\} =  \sigma(x)\cup\{0\}\quad\mbox{ for all $x\in A$,}$$
  then $e=1$.
\end{corollary}

\begin{proof}
Theorem \ref{T1} says that $e$ is an idempotent. By taking $1-e$ for $x$ we obtain $e=1$.  
\end{proof}

\subsection{The unit-regular element case}
An element of a ring $R$ that can be written as the product of an idempotent and an invertible element is called a {\em unit-regular element}.  We say that $R$ is a {\em unit-regular ring}
if all its elements are unit-regular. Unit-regularity is an old and thoroughly studied concept in Ring Theory.

\begin{theorem}\label{T}
Let  $A$ be a   semisimple Banach algebra and  let $a,b\in A$ be such that $\sigma(ax) = \sigma(bx)$ for all $x\in A$. If $a$ is a unit-regular element,  then $a=b$.
\end{theorem}

% \begin{corollary}\label{Te}
% Let  $A$ be a   semisimple Banach algebra and  let $e,b\in A$ satisfy $\sigma(ex) = \sigma(bx)$ for all $x\in A$. If $e$ is an idempotent,  then $e=b$. \end{corollary}
 
 \begin{proof} 
 We have $a = et$ with $e$ an idempotent and $t$ invertible. Replacing $x$ by $t^{-1}x$ in $\sigma(ax) = \sigma(bx)$ we get $\sigma(ex) = \sigma(b'x)$ for all $x\in A$, where $b'=bt^{-1}$. Hence we see that with no loss of generality we may assume that $a=e$ is an idempotent. Further, in view of Corollary \ref{C1} we may also assume that $e\ne 1$.
 
 Replacing $x$ by $(1-e)x$ in  $\sigma(ex) = \sigma(bx)$ we get $ \sigma\bigl(b(1-e)x\bigr) =0$,
 and therefore $b(1-e)=0$ by the semisimplicity of $A$.
 Similarly, replacing $x$ by $x(1-e)$ we get $\sigma\bigl(ex(1-e)\bigr) = \sigma\bigl(bx(1-e)\bigr)$, hence $\sigma\bigl((1-e)ex\bigr)\cup \{0\} 
 = \sigma\bigl((1-e)bx\bigr)\cup \{0\}$, which gives
 $\sigma\bigl((1-e)bx\bigr) = \{0\}$. Consequently, $(1-e)b =0$. Together with $b(1-e)=0$ this yields $b \in eAe$.
 
 It is easy to see that $eAe$ is a Banach subalgebra of $A$ with $e$ as an identity element, and that $\sigma_A(y)= \sigma_{eAe}(y)\cup\{0\}$ for every 
 $y\in eAe$ (see, e.g., \cite[Theorem 1.6.15]{Ric}).  The condition $\sigma(exe) = \sigma(b\cdot exe)$ for every $x\in A$ can be therefore rewritten as
 $\sigma_{eAe}(y)\cup\{0\} = \sigma_{eAe}(by)\cup\{0\}$  for every $y\in eAe$. Since the algebra $eAe$ is also semisimple, we infer from Corollary 
 \ref{C1} that $b=e$.
 \end{proof}

\subsection{The commutative case} Relying on known results,   Problem \ref{P} can be easily settled in the commutative case.

\begin{theorem}
If $A$ is a  commutative semisimple Banach algebra and $a,b\in A$ satisfy $\sigma(ax) = \sigma(bx)$
for all $x\in A$, then $a=b$.
\end{theorem}
\begin{proof}
By the Gelfand representation theorem we may consider $A$ as a subalgebra of $C(K)$, the algebra of all continuous functions on a compact  Hausdorff space $K$, which 
separates points and contains constants. Thus its closure $\overline{A}$ with respect to the uniform norm is a uniform algebra.
Since the spectrum in commutative Banach algebras is continuous \cite[Theorem 3.4.1]{Ab},  $\sigma(ax)=\sigma(bx)$ holds for all $x\in \overline{A}$. Therefore $a=b$ follows from   \cite[Lemma 3]{LT}. 
\end{proof}

\subsection{The $C^*$-algebra case} 
We consider the next theorem as the main result  of this section.
\begin{theorem}\label{Sp}
If $A$ is a  $C^*$-algebra and $a,b\in A$ satisfy $\sigma(ax) = \sigma(bx)$ for all $x\in A$, then $a=b$.
\end{theorem}

\begin{proof} The proof is divided into four steps.
\smallskip 
%Let us first prove a reduced form of the theorem, namely we will assume that $a=a^*$.

{\bf Claim 1:} {\em If $a=a^*$, then $b=b^*$}.
\smallskip 

%We begin by proving that $a=a^*$ implies $b=b^*$.
On the contrary, suppose that $b-b^*\neq 0$. Take an irreducible representation $\pi$ of $A$ on a Hilbert space $H$  such that $\pi(b-b^*)\neq 0$, i.e.,  
 $\pi(b)$ is not self-adjoint. 
Then there exists $\xi\in H$, $||\xi||=1$, such that $\alpha= \langle\pi(b)\xi,\xi\rangle\in \CC \setminus \R$. 
Then $\eta=  \pi(b)\xi-\alpha \xi$  satisfies $\langle \eta,\xi \rangle=0$. 
By Kadison's transitivity theorem (see, e.g., \cite[Theorem 5.2.2]{Mur}) there 
exists $t\in A$ such that  $\pi(t)\xi=\xi$, $\pi(t)\eta=0$, and $t=t^*$. 
Therefore $\pi(t)\pi(b)\pi(t)\xi=\alpha \xi$, which gives
 $$\alpha\in \s\bigl(\pi(t)\pi(b)\pi(t)\bigr)\subseteq \s(tbt)\subseteq \s(bt^2) \cup \{0\}=\s(at^2)\cup \{0\}=\s(tat)\cup \{0\}.$$ 
This is a contradiction since  $tat$ is self-adjoint and so its spectrum  contains  only real numbers. 

\smallskip
{\bf Claim 2:} {\em If $a=a^*$, then $ab = ba$}.
\smallskip 

Replacing $x$ by $ax$ in  $\s(ax)=\s(bx)$ we get  
$\s(a^2x) = \s(bax)$ for every $x\in A$. Since $a^2$ is self-adjoint, Claim 1 implies that $ba$ is self-adjoint, too. Since $b$ is also self-adjoint
by Claim 1, it follows that  $ab = ba$.

%From the preceding paragraph we deduce that  $(ax)^*=ax$ if and only if  $(bx)^*=bx$. 
%When $x^*=x$ this amounts to $[a,x]=0$ if and only if $[b,x]=0$.
%This further implies that the same is true for an arbitrary $x\in A$. In particular, $a$ and $b$ commute.

\smallskip
{\bf Claim 3:} {\em If $a=a^*$, then $a=b$.}
\smallskip 

Claims 1 and 2 imply that the   $C^*$-subalgebra of $A$  generated by  $a$ and $b$ is  commutative.
Since $\sigma(ax) = \sigma(bx)$ of course holds for every $x$ from this subalgebra, there is no loss of generality in assuming that 
$A$ is commutative. Thus it suffices to treat the case where $A=C(K)$, the algebra of all continuous functions on a compact Hausdorff space $K$. Suppose
$a\neq b$. Then there exists an open subset $U\subseteq K$ such that $a(U)\cap b(U)=\emptyset$.  
Without loss of generality it can be assumed that $\sup_{x\in U} |a(x)|\geq \sup_{x\in U} |b(x)|$. Choose  $x_0\in U$ such that  $|a(x_0)|\ge  \sup_{x\in U} |b(x)|$.  
We can apply Urysohn's lemma to obtain a continuous function $h:K\to [0,1]$ 
with ${\rm supp}(h)\subseteq U$ and $h(x_0)=1$. Hence
$a(x_0)\not \in bh(U)=bh(K)$, and therefore $ah(x_0)=a(x_0)\not \in \s(bh)$, contrary to our assumption.

\smallskip
{\bf Claim 4:} {\em If $a$ is arbitrary, then  $a=b$.}
\smallskip 
 
As special cases of  $\s(ax)=\s(bx)$, $x\in A$, we have 
 $\s(aa^*x)=\s(ba^*x)$, $x\in A$, and $\s(ab^*x)=\s(bb^*x)$, $x\in A$. 
From Claim 3 we infer that $aa^*=ba^*$ and $ab^*=bb^*$. Accordingly,
 $(a-b)(a^*-b^*)=0$, which results in $a=b$.
\end{proof}

\section{The condition $r(ax) \le r(bx)$} \label{S3}

 What to expect if elements $a$ and $b$ from a semisimple Banach algebra $A$ satisfy \eqref{sr}? An obvious possibility is that there exists  $u\in Z(A)$ such that $r(u) \le 1$ and 
$a=ub$. In fact, $u$ does not need to be central, it is enough to assume that it commutes with all elements from the right ideal $bA$. We shall see that, unfortunately, the possibility $a=ub$ is not the only one in general; however, in two interesting special cases it is.

\subsection{The invertible element case}
If $b$ is invertible, then the solution to our problem follows immediately from  Ptak's result \cite{Ptak}.

\begin{theorem} \label{Tpt}
Let $A$ be a semisimple Banach algebra and let $a,b\in A$ be such that  $r(ax) \le r(bx)$ for all $x\in A$. If $b$ is invertible, then there  exists $u\in Z(A)$ such that $r(u) \le 1$ and 
$a=ub$.
\end{theorem}
 
\begin{proof} 
Set $u=ab^{-1}$.
Our assumption can be written as  $r(ux) \le r(x)$ for all $x\in A$. Hence $u\in Z(A)$
by \cite[Proposition 2.1]{Ptak} (see also \cite[Theorem 2.2]{BBR}). Letting $x=1$ we get $r(u)\le 1$.
\end{proof}

\subsection{Remarks on the $C^*$-algebra case}
From now on we confine ourselves to $C^*$-algebras. We begin with a useful rewording of condition \eqref{sr}.

\begin{lemma}\label{Lr}
Let $A$ be a $C^*$-algebra and let $a,b\in A$. The following conditions are equivalent:
\begin{enumerate}
\item[(i)] $r(ax)\le r(bx)$ for all $x\in A$.
\item[(ii)] $||yaz||\le ||ybz||$ for all $y,z\in A$. 
\end{enumerate}
\end{lemma}

\begin{proof} (i)$\Longrightarrow$(ii). Note that  $yaz$ and $ybz$ satisfy the same condition as $a$ and $b$, that is,
$$
r(yaz\cdot x) = r(azxy) \le r(bzxy) = r(ybz\cdot x).
$$
Therefore it suffices to show that (i) implies   $||a||\le ||b||$. And this is easy:
$$
||a||^2 = r(aa^*)\le r(ba^*) = r\bigl((ba^*)^*\bigl)= r(ab^*)\le r(bb^*)=||b||^2.
$$

(ii)$\Longrightarrow$(i). From (ii) we infer that
\begin{eqnarray*}
||(ax)^n|| &=& ||axaxax\ldots ax||\le ||bxaxax\ldots ax|| \\
&\le &||bxbxax\ldots ax|| \le 
 \ldots \le  ||bxbxbx\ldots bx|| \\
 &=&   ||(bx)^n||. 
\end{eqnarray*}
Therefore (i) follows from the spectral radius formula.
\end{proof} 

The following simple example indicates the delicacy of our problem.

\begin{example}
Let $A$ be the commutative $C^*$-algebra $C[-1,1]$, and let $a,b\in A$ be given by $a(t)=t$, $b(t)=|t|$. Then $$r(ax)=||ax||=||bx||= r(bx)\quad\mbox{ for all $x\in A$.}$$
 However, there does not exist $u\in A$ such that $a=ub$.
\end{example}

This  suggests that in order to derive $a = ub$ with $u\in Z(A)$ from \eqref{sr} it might be reasonable  to consider $C^*$-algebras whose center is small. In what follows we will deal with
{\em prime $C^*$-algebras}, i.e., $C^*$-algebras with the property that that the product of any two of their nonzero ideals is nonzero.  This is a fairly large class of $C^*$-algebras, which includes all primitive ones. It is known that such algebras have trivial centers, i.e., scalar multiples of $1$ are their only central elements.
Also, it is easy to see that only these elements  commute with every element from a nonzero right ideal.    

\subsection{Tools} In the course of the proof we will use several tools which are not standard in spectral theory. For the clarity of the exposition  we will therefore  state them  as lemmas. The first one is of crucial importance for our goal.

\begin{lemma} \label{Lzp}
Let $B$ be a $C^*$-algebra and let $X$ be a Banach space. If $\Phi:B\times B\to X$ is a continuous bilinear map such that $\Phi(y,z)=0$ whenever $y,z\in B$ satisfy $yz=0$, then 
$\Phi(yx,z)=\Phi(y,xz)$ for all $x,y,z\in B$.
\end{lemma}

\begin{proof}
This result actually holds for a large class of Banach algebras which includes $C^*$-algebras; see \cite[Theorem 2.11 and Example 2, p. 137]{ABEV}. 
\end{proof}

\begin{lemma} \label{Lec}
Let $A$ be a prime $C^*$-algebra. Suppose $a,b,c,d\in A$ satisfy $axb = cxd$ for all $x\in A$. If $a\ne 0$, then  $b\in \mathbb C d$. Similarly, if $b\ne 0$, then $a\in \mathbb C c$.
\end{lemma}

\begin{proof}
This result is basically due to Martindale \cite{Mart} and it actually holds for general prime rings, just that $\mathbb C$ must be replaced by the so-called extended centroid (a certain extension of the center). It is a fact that the extended centroid  of a prime $C^*$-algebra is equal to $\mathbb C$ \cite[Proposition 2.2.10]{AM}. 
\end{proof}

In the next lemma we consider a special {\em functional identity} which can be handled by elementary means, avoiding the general theory \cite{FIbook}.   At the beginning of the proof we will use an idea from \cite[Example 1.4]{FIbook}.

\begin{lemma} \label{Lfi}
Let $A$ be a prime $C^*$-algebra. Suppose there exist a map $f:A\to A$ and $c\in A$ such that
 $$f(x)yc + f(y)xc=0\quad\mbox{ for all $x,y\in A$.}$$
 If $f\ne 0$ and $c\ne 0$,  then there exists a faithful irreducible representation $\pi$ of $A$ on a Hilbert space $H$ such that $\pi(A)$ contains $K(H)$, the algebra of all compact operators on $H$.
\end{lemma}

\begin{proof}
 Our assumption implies that for all $x,y,z\in A$ we have 
$$
f(y)xczc=- f(x)yczc = f(ycz)xc.
$$
Fixing $y\in A$ such that $f(y)\ne 0$ we infer from  Lemma \ref{Lec} that for every $z\in A$ there exists $\lambda_z\in \mathbb C$ such that $czc = \lambda_z c$.  Consequently, $(c^*c)^2 = \alpha c^*c$ for some  $\alpha\in\mathbb R\setminus{\{0\}}$. Note that $e = \alpha^{-1}c^*c$ satisfies $e^2 = e = e^*$ and $eAe=\mathbb C e$, so we may identify $eAe$ with $\mathbb C$. We endow $Ae$ with an inner product $\langle ae, be\rangle=eb^*ae$. Note that the inner product norm  coincides with the original norm on $Ae$, and is therefore complete.
We  denote the corresponding Hilbert space by $H$. Define $\pi:A\to B(H)$ according to $\pi(a)\xi = a\xi$, $a\in A$, $\xi\in H$, and note that $\pi$ is an irreducible representation of $A$ on $H$. Moreover, a faithful one since $A$ is prime. Since $\pi(e)$ is a rank one operator, it follows that $K(H)\subseteq \pi(A)$ (see, e.g., \cite[Theorem 2.4.9]{Mur}).
 \end{proof}

%The latter implies that 
%$I=Ae$ is a minimal left ideal of $A$. This is well-known, but let us give a proof for the sake of completness. Take a nonzero left ideal $L$ of $A$ such that
%$L\subseteq I$. Pick $0\ne \ell\in L$. Since $A$ is prime there exists $y\in A$ such that $ey\ell\ne 0$.  As $\ell \in L\subseteq Ae$, we have $ey\ell = \mu e$ for some  $\mu\in \mathbb C\setminus{\{0\}}$.
%Therefore $c = \mu^{-1}cy\ell \in L$, which yields $I\subseteq L$. Therefore $I$ is indeed minimal. Hence there exists  a projection $p=p^*\in A$ such that $I=Ap$ and $pAp=\mathbb Cp$
 %\cite[BA.4.3]{Bar}, so we can identify $pAp$ with $\mathbb C$. We endow $Ap$ with an inner product $\langle ap, bp\rangle=pb^*ap$ and denote the corresponding Hilbert space by $H$. (Note that the scalar product norm on $Ap$ coincides with its original norm, which asserts completeness of $H$.) Now $A$ can be embedded in $B(H)$ via left multiplication on $H=Ap$. It is easy to check that this embedding is $*$-homomorphism. It is also injective because $A$ is prime. Moreover, minimality of $Ap$ asserts that $A$ acts irreducibly on $H$. Thus, we have an irreducible faithful representation $\pi$ of $C^*$-algebra $A$ into $B(H)$ with $\pi(p)\in K(H)\cap \pi(A)\neq \{0\}$.  Using \cite[Theorem 2.4.9]{Mur} we deduce that $K(H)\subseteq \pi(A)$.  

\subsection{The prime $C^*$-algebra case} We now have enough information to prove the main result of this section.

\begin{theorem}\label{Tr}
Let $A$ be a prime $C^*$-algebra and let $a,b\in A$ be such that  $r(ax) \le r(bx)$ for all $x\in A$. Then there  exists $\lambda\in \mathbb C$ such that $|\lambda| \le 1$ and 
$a=\lambda b$.
\end{theorem}
 
 \begin{proof} Obviously  it suffices to prove that $a\in \mathbb C b$. 
We divide the proof into four steps.
\smallskip 
%Let us first prove a reduced form of the theorem, namely we will assume that $a=a^*$.

{\bf Claim 1:} {\em If $b=b^*$, then $ab=ba$}.

\smallskip 
 Let $B$ be the $C^*$-algebra generated by $b$. Define $\Phi:B\times B\to A$ by $\Phi(y,z)=yaz$. Since $B$ is commutative, $yz=0$ implies $ybz=0$. According to Lemma \ref{Lr} this further gives
 $\Phi(y,z)=0$. Lemma \ref{Lzp} therefore tells us that $\Phi(yx,z) = \Phi(y,xz)$  for all $x,y,z\in B$. Setting $y=z=1$ and $x=b$ we get $ab=ba$.

\smallskip

Define $f:A\to A$ by $f(x)=axb^*b - bxb^*a$.

\smallskip

{\bf Claim 2:} {\em $f(x)yb^* + f(y)xb^* =0$ for all $x,y\in A$.}
\smallskip 

Take a self-adjoint $s\in A$.
Substituting $sb^*x$ for $x$ in $r(ax) \le r(bx)$ we get $r(asb^*x)\le r(bsb^*x)$ for every $x\in A$.  Since $bsb^*$ is self-adjoint, Claim 1 implies that $(asb^*)(bsb^*) = (bsb^*)(asb^*)$, i.e.,
$f(s)sb^* =0$ holds 
 for an arbitrary self-adjoint $s\in A$. Replacing $s$ by $s+t$ with both $s,t$ self-adjoint it follows that $f(s)tb^* + f(t)sb^* =0$. Since every element in $A$ is a linear combination of two self-adjoint elements, the desired conclusion follows.
 
\smallskip

{\bf Claim 3:} {\em If $f\ne 0$, then $a\in \mathbb C b$.}

\smallskip 

Lemma \ref{Lfi} says that there exists a faithful representation $\pi$ of $A$ on a Hilbert space $H$ such that $K(H)\subseteq \pi(A)$.   By $\xi\otimes \eta$ we denote the rank one operator given by $(\xi\otimes \eta)\omega = \langle\omega,\eta\rangle \xi$. Note that
$\sigma_{B(H)}(\xi\otimes \eta) = \{0, \langle \xi,\eta\rangle\}$ and that $A(\xi\otimes \eta) = A\xi\otimes \eta$ for every $A\in B(H)$.
 Of course, $\xi\otimes \eta\in \pi(A)$, and hence
  $$r\bigl(\pi(a)(\xi\otimes \eta)\bigr)\le r\bigl(\pi(b)(\xi\otimes \eta)\bigr).$$
   That is, 
 $$
 |\langle \pi(a)\xi,\eta\rangle| \le |\langle \pi(b)\xi,\eta\rangle|, 
 $$
 where $\xi$ and $\eta$ are arbitrary vectors in $H$. If $\pi(a)\xi$ was not a scalar multiple of  $\pi(b)\xi$, then we could find $\eta$ such that $\langle \pi(a)\xi,\eta\rangle\ne 0$ and 
 $\langle \pi(b)\xi,\eta\rangle=0$ -- a contradiction.  Therefore $\pi(a)\xi\in \mathbb C \pi(b)\xi$ for every $\xi\in H$, hence  $\pi(a)\in \mathbb C \pi(b)$ by Lemma \ref{Lld}, and so $a\in \mathbb C b$.
 
\smallskip

{\bf Claim 4:} {\em If $f= 0$, then $a\in \mathbb C b$.}

\smallskip 
The result is trivial if $b=0$, so let $b\ne 0$. We are assuming that $axb^*b = bxb^*a$ holds for every $x\in A$. Since $b^*b\ne 0$, we have $a\in \mathbb C b$ by Lemma \ref{Lec}.
 \end{proof}

\begin{corollary}\label{TrC}
Let $A$ be a prime $C^*$-algebra and let $a,b\in A$ be such that  $r(ax) = r(bx)$ for all $x\in A$. Then there  exists $\lambda\in \mathbb C$ such that $|\lambda| = 1$ and 
$a=\lambda b$.
\end{corollary}

\section{Spectral characterizations of multiplicative maps}

Let $A_0$ and $A$ be Banach algebras, and let $\varphi:A_0\to A$ be a surjective linear map such that 
\begin{equation}\label{e2}
\sigma\bigl(\varphi(x)\bigr) = \sigma(x) \quad\mbox{for all $x\in A_0$.}
\end{equation}  
 Under what conditions   $\varphi$ is a Jordan homomorphism?
This is a classical problem in the Banach algebra theory, initiated by Kaplansky in \cite{K}.
It is expected that a sufficient condition is that $A$ is a $C^*$-algebra, or maybe even a general semisimple Banach algebra. In spite of considerable efforts of numerous authors, the problem seems to be out of reach at such level of generality; see, e.g., \cite{BS2} for historic comments. One is therefore inclined to consider modifications of  \eqref{e2} that can be handled and may give some light on 
the classical situation. In \cite{Mol} Molnar described not necessarily linear surjective maps $\varphi$ satisfying 
\begin{equation}\label{e1}
\sigma\bigl(\varphi(x)\varphi(y)\bigr) = \sigma(xy) \quad\mbox{for all $x,y\in A_0$}
\end{equation}
in the case where $A_0=A=B(H)$  or $A_0=A=C(K)$. These results have been extended in different directions (see  \cite{HLW, LT, TL} and references therein), but these generalizations also deal only with 
 some special algebras. It  seems that it is not easy to treat \eqref{e1} in general classes of algebras. We will consider similar, but more easily approachable conditions \eqref{es} and \eqref{er}.
 Using the results of the previous sections we will be able to  handle them in quite general algebras.

\subsection{The condition $\sigma\bigl(\varphi(x)\varphi(y)\varphi(z)\bigr) = \sigma(xyz)$ } We begin with an application of Theorem \ref{T}.
%This can be very useful when studying spectral properties.

\begin{corollary}\label{LLL} Let $A_0$ and $A$ be Banach algebras with $A$ semisimple. Let $\varphi:A_0\to A$ be a surjective map satisfying $\sigma\bigl(\varphi(x)\varphi(y)\varphi(z)\bigr) = \sigma(xyz)$ for all
$x,y,z\in A_0$. Then $\varphi(1)\in Z(A)$, $\varphi(1)^3 =1$, and $\varphi(xy) = \varphi(1)^2 \varphi(x)\varphi(y)$ for all invertible $x,y\in A_0$.
\end{corollary}

\begin{proof} Set $u=\varphi(1)$.
Taking $x=y=z=1$ we get $ \sigma(u^3) = \{1\}$. In particular, $u$ is invertible. Next we have
$$
\sigma\bigl(u\varphi(y)\varphi(z)\bigr) = \sigma(1yz) = \sigma(y1z) = \sigma\bigl(\varphi(y)u\varphi(z)\bigr) 
$$
for all $x,y\in A_0$. From Theorem \ref{T} it follows that $\varphi(y)u= u\varphi(y)$ whenever $\varphi(y)$ is invertible. That is, $u$ commutes with all invertible elements in $A$, and is therefore contained in $Z(A)$.  Hence $u^3$ also belongs to $Z(A)$, and so $ \sigma(u^3) = \{1\}$ implies that $u^3=1$. 

From $\sigma(u^2\varphi(y)) = \sigma(y)$ we see that $\varphi(y)$ is invertible whenever $y$ is invertible. Take invertible $x,y\in A_0$. Applying Theorem \ref{T} to
$$
\sigma\bigl(\varphi(x)\varphi(y)\varphi(z)\bigr) = \sigma(xyz)= \sigma\bigl(1(xy)z\bigr)= \sigma\bigl(u\varphi(xy)\varphi(z)\bigr)
$$
we thus get $\varphi(x)\varphi(y) = u\varphi(xy)$.
\end{proof}

Adding the assumption that $\varphi$ is linear we get a definitive conclusion.

\begin{corollary}\label{ce3}
Let $A_0$ and $A$ be Banach algebras with $A$ semisimple. Let $\varphi:A_0\to A$ be a surjective linear map satisfying $\sigma\bigl(\varphi(x)\varphi(y)\varphi(z)\bigr) = \sigma(xyz)$ for all
$x,y,z\in A_0$. Then $\varphi(1)\in Z(A)$, $\varphi(1)^3 =1$, and $\varphi(xy) = \varphi(1)^2 \varphi(x)\varphi(y)$ for all  $x,y\in A_0$.
\end{corollary}

\begin{proof}
 If $x\in A_0$ is arbitrary, then $x - \lambda 1$ is invertible for some $\lambda\in\mathbb C$, and so $\varphi\bigl((x-\lambda 1)y\bigr) = \varphi(1)^2\varphi(x-\lambda 1)\varphi(y)$ for every invertible $y$. As $\varphi$ is linear this clearly yields $\varphi(xy) = \varphi(1)^2 \varphi(x)\varphi(y)$. A similar argument shows that the same is true if $y$ is not invertible.
\end{proof}

In the  $C^*$-algebra case  we do not need to assume the linearity, which brings us closer to Molnar's results \cite{Mol}.

\begin{corollary}\label{ce3c*}
Let $A_0$  be a Banach algebra and $A$ be a $C^*$-algebra. Let $\varphi:A_0\to A$ be a surjective  map satisfying $\sigma\bigl(\varphi(x)\varphi(y)\varphi(z)\bigr) = \sigma(xyz)$ for all
$x,y,z\in A_0$. Then $\varphi(1)\in Z(A)$, $\varphi(1)^3 =1$, and $\varphi(xy) = \varphi(1)^2 \varphi(x)\varphi(y)$ for all  $x,y\in A_0$.
\end{corollary}

\begin{proof} 
The same argument as in the proof of Corollary \ref{LLL} works, except that at the end we  may take arbitrary $x$ and $y$ and  then apply Theorem \ref{Sp} instead of Theorem \ref{T}.
\end{proof}

Our conclusion can be read as that the map $x\mapsto \varphi(1)^2 \varphi(x)$ is multiplicative.  We remark that multiplicative maps on rings often turn out to be automatically additive \cite{Mart2, Ric1}; say, this is true in prime rings having nontrivial idempotents. Accordingly, by adding some assumptions to Corollary \ref{ce3c*} one can get a more complete result.  See also \cite{Mol}.

\subsection{The condition $r\bigl(\varphi(x)\varphi(y)\varphi(z)\bigr) = r(xyz)$ } Our final result is a corollary to Theorem \ref{Tr}.

\begin{corollary}\label{cec*}
Let $A_0$  be a Banach algebra and $A$ be a prime $C^*$-algebra. Let $\varphi:A_0\to A$ be a surjective  map satisfying $r\bigl(\varphi(x)\varphi(y)\varphi(z)\bigr) = r(xyz)$ for all
$x,y,z\in A_0$. Then for each pair $x,y\in A_0$ there exists $\lambda(x,y)\in \mathbb C$ such that $|\lambda(x,y)|=1$ and  $\varphi(xy) = \lambda(x,y) \varphi(x)\varphi(y)$.
\end{corollary}

\begin{proof}
We argue similarly as in the proof of Corollary \ref{LLL}.  Set $u=\varphi(1)$. For all $y,z\in A_0$ we have
$$
r\bigl(u\varphi(y)\varphi(z)\bigr) = r(1yz) = r(y1z) = r\bigl(\varphi(y)u\varphi(z)\bigr) 
$$
Corollary \ref{TrC} tells us that $u\varphi(y)$ and $\varphi(y)u$ are equal up to a scalar factor of modulus 1. Thus, for every $x\in A$ there exists $\mu_x\in \mathbb C$ such that $|\mu_x|=1$ and
$ux = \mu_x xu$. Hence
$$
\mu_{x+1} (x+1)u = u(x+1) = ux + u = \mu_{x} xu + u 
$$
for every $x\in A$. That is,
$$
(\mu_{x+1} - \mu_x)xu = (1-\mu_{x+1})u.
$$
Therefore either $xu\in \mathbb C u$ or $\mu_{x+1} = 1$, i.e., $ux = xu$. In each of the two cases we have $uxu= xu^2$. Lemma \ref{Lec} implies that $u\in \mathbb C$ (and so we can actually take $\mu_x=1$ for every $x\in A$).  From $r(u^3)=1$ we see that $|u|=1$. Finally, we have
$$
r\bigl(\varphi(x)\varphi(y)\varphi(z)\bigr) = r(xyz)= r\bigl(1(xy)z)\bigr)= r\bigl(u\varphi(xy)\varphi(z)\bigr)
$$
and so the desired conclusion follows from Corollary \ref{TrC}.
\end{proof}

\begin{remark}
The scalars $\lambda(x,y)$ are not entirely arbitrary. From $\varphi\bigl((xy)z\bigr) = \varphi\bigl(x(yz)\bigr)$ one immediately infers that 
\begin{equation}\label{ha}
\lambda(xy,z)\lambda(x,y) = \lambda(x,yz)\lambda(y,z),
\end{equation}
unless $\varphi(x)\varphi(y)\varphi(z)=0$. In Group Theory,  maps satisfying \eqref{ha} are called 2-cocycles. Their homology classes form
the second cohomology group. Since any further discussion in this
direction would lead us too far from the scope of this paper, let us just
say that standard results from Homological Algebra indicate that finding a more detailed description of $\lambda(x,y)$ may be a very difficult task. \end{remark}

{\bf Acknowledgements}. We are thankful to Lajos Molnar for drawing our attention to \cite{LT}, and to Primo\v z Moravec for providing us with relevant information from Homological Algebra.

\end{document}